\documentclass[reqno]{amsart}
\usepackage{amsthm, amssymb, amsfonts}
\usepackage{lmodern}
\usepackage{color}
\usepackage[T5]{fontenc}
\usepackage{amscd,amssymb}
\usepackage[v2,cmtip]{xy}
\usepackage{mathrsfs}
\theoremstyle{plain}
\newtheorem{thm}{Theorem}[section]
\newtheorem{lem}[thm]{Lemma}

\newtheorem{prop}[thm]{Proposition}
\theoremstyle{definition}

\newtheorem{nota}[thm]{Notation}

\usepackage[left=4.8cm, right=4.8cm, top=4.5cm, bottom=4.5cm]{geometry}

\begin{document}

\setlength{\baselineskip}{16pt}
 
\title[Steenrod algebra on the modular invariants]{On the action of the Steenrod algebra on the modular invariants of special linear group}

 \author{Nguy\~\ecircumflex n Sum}

\address{Department of Mathematics, Quy Nh\ohorn n University, 
170 An D\uhorn \ohorn ng V\uhorn \ohorn ng, Quy Nh\ohorn n, B\`inh \DJ\d inh, Viet Nam}
\email{nguyensum@qnu.edu.vn}

\begin{abstract} 
We  compute the action of the Steenrod algebra on generators of algebras of  invariants of special linear group ${SL_n=SL(n, \mathbb Z/p)}$ in the polynomial algebra with $ p$ an odd prime number. 
\end{abstract}

\maketitle

\section{Introduction}

For an an odd prime $p$, let $SL_n$ denote the special linear subgroup of $GL(n,\mathbb Z/p)$, which acts naturally on the cohomology algebra $H^*(B(\mathbb Z/p)^n)$. Here and in what follows, the cohomology is always taken with coefficients in the prime field $\mathbb Z/p$. 

According to \cite{mui2}, $H^*(B(\mathbb Z/p)^n) = E(x_1,\ldots , x_n)\otimes P(y_1,\ldots, y_n)$ with $\dim x_i = 1$, $y_i = \beta x_i$, where $\beta$ is the Bockstein homomorphism, $E(.,\ldots,.)$ and $P(.,\ldots,.)$ are the exterior and polynomial algebras over $\mathbb Z/p$ generated by the variables indicated. Let $(e_{k+1},\ldots,e_n), k \geqslant 0,$be a sequence of non-negative integers. Following M\`ui \cite{mui1}, we define
$$ [k;e_{k+1},\ldots,e_n] = [k;e_{k+1},\ldots,e_n](x_1,\ldots , x_n,y_1,\ldots, y_n)$$
by 
$$ [k;e_{k+1},  \ldots,  e_n]  =  \frac 1{k!}
\begin{vmatrix} x_1&\cdots &x_n\\
  \vdots&\cdots &\vdots\\
  x_1&\cdots &x_n\\
  y_1^{p^{e_{k+1}}}&\cdots &y_n^{p^{e_{k+1}}}\\
  \vdots&\cdots  &\vdots\\
   y_1^{p^{e_n}} & \cdots & y_n^{p^{e_n}}
   \end{vmatrix}. $$
The precise meaning of the right hand side is given in \cite{mui1}. For $k= 0$, we write
$$[0;e_{1},\ldots,e_n] = [e_{1},\ldots,e_n] = \det \left(y_i^{p^{e_j}}\right).$$
We set
\begin{align*} L_{n,s} &= [0,\ldots, \hat s,\ldots,n], \ 0 \leqslant s \leqslant n,\\ L_n &= L_{n,n} = [0,\ldots,n-1].
\end{align*}
Each $[k;e_{k+1},  \ldots,  e_n]$ is an invariant of $SL_n$ and $[e_{1},\ldots,e_n]$ is divisible by $L_n$. Then Dickson invariants $Q_{n,s}, \ 0 \leqslant s \leqslant n,$ and M\`ui invariants $ M_{n,s_1,\ldots,s_k}$, $0 \leq \ s_1 < \ldots < s_k \le n$ are defined by
\begin{align*} Q_{n,s}&=L_{n,s}/L_n,\\ M_{n,s_1,\ldots,s_k} &= [k;0,\ldots,\hat s_1,\ldots,\hat s_k,\ldots,n-1].\end{align*}
Note that $Q_{n,n} =1$, $Q_{n,0} = L_n^{p-1}$, $M_{n,0,\ldots,n-1} = [n;\emptyset] = x_1\ldots x_n$. 

M\`ui proved in \cite{mui1} that $H^*(B(\mathbb Z/p)^n)^{SL_n}$ is the free module over the Dickson algebra $P(L_n,Q_{n,1},\ldots,Q_{n,n-1})$ generated by 1 and $ M_{n,s_1,\ldots,s_k}$ with $0 \leq \ s_1 < \ldots < s_k \le n$.

\medskip
The Steenrod algebra $\mathcal A(p)$ acts on $H^*(B(\mathbb Z/p)^n)$ by well-known rules. Since this action commutes with the action of $SL_n$, it induces an action of $\mathcal A(p)$ on $H^*(B(\mathbb Z/p)^n)^{SL_n}$.

Let $\tau_s$ and $\xi_i$ be the Milnor elements of dimensions $2p^s-1$ and $2p^i-2$, respectively, in the dual algebra $\mathcal A(p)^*$ of $\mathcal A(p)$. Milnor showed in \cite{mil} that
$$\mathcal A(p)^* = E(\tau_0,\tau_1, \ldots )\otimes P(\xi_1, \xi_2,\ldots).$$ 
So, $\mathcal A(p)^*$ has a basis consisting of all monomials $\tau_S\xi^R = \tau_{s_1}\ldots \tau_{s_t}\xi_1^{r_1}\ldots \xi_m^{r_m}$ with $S = (s_1,\ldots,s_t)$, $0\leqslant s_1 < \ldots < s_t$, $R= (r_1,\ldots,r_m)$. Let $St^{S,R} \in \mathcal A(p)$ denote the dual of $\tau_S\xi^R$ with respect to this basis of $\mathcal A(p)^*$. Then $\mathcal A(p)$ has a new basis consisting of all operations $St^{S,R}$. In particular, for $S = \emptyset$, $R=(k)$, $St^{S,R}$ is nothing but the Steenrod operation $P^k$.

The action of $P^k$ on Dickson and M\`ui invariants was explicitly computed by H\uhorn ng and Minh \cite{hum}. The action of $St^{S,R}$ on the invariant $[n,\emptyset] =  x_1\ldots x_n$ was computed by M\`ui \cite{mui2}.

\medskip
In this paper, we compute the action of $St^{S,R}$ on  $[k;e_{k+1},  \ldots,  e_n]$ and prove a nice relation between the invariants $[k;e_{k+1},  \ldots,  e_n+s]$, $0\leqslant s \leqslant n$, and the Dickson invariants. Using these results, we explicitly compute the action of $P^k$ on M\`ui invariants $ M_{n,s_1,\ldots,s_k}$ which was first computed in H\uhorn ng and Minh \cite{hum} by another method.

\medskip
To state the main results, we introduce some notations. Let $J = (J_0,J_1,$ $\ldots ,J_m)$ with $J_s \subset \{{k+1},  \ldots,  n\}$, $0 \leqslant s \leqslant m$, and $\coprod_{s=0}^mJ_s = \{{k+1},  \ldots,  n\}$ (disjoint union). We define the sequence $R_J = (r_{J_1},\ldots,r_{J_m})$ and the function $\Phi_J:\{{k+1},  \ldots,  n\} \to \{{1},  \ldots,  m\}$ by setting
\begin{align*} r_{J_s} &= \sum_{j\in J_s} p^{e_j},\ 0\leqslant s \leqslant m,\\
\Phi_J(i) &= s \ \text{ if }\ i \in J_s,\ k+1 \leqslant i \leqslant n.
\end{align*}
The main result of this paper is 
\begin{thm}\label{dl1} Suppose that $e_i \ne e_j$ for $i \ne j$, $S = (s_1,\ldots,s_t)$, $s_1 < \ldots < s_t<m$. Under the above notation we have
\begin{multline*}St^{S,R}[k;e_{k+1},  \ldots,  e_n]\\ = \begin{cases} (-1)^{t(k-t)}[k-t,s_1,\ldots,s_t,e_{k+1}+\Phi_J(k+1),\ldots , e_n + \Phi_J(n)],\\ \hskip5.5cm R= R_J, \ \text{ for some } J,\\ 0, \hskip5cm \text{ otherwise. }
\end{cases}
\end{multline*}
\end{thm}
 We have also the following relation from which we can explicitly compute $St^{S,R}[k;e_{k+1},  \ldots,  e_n]$ in terms of Dickson and M\`ui invariants.
\begin{prop}\label{md2} For $0\leqslant k \leqslant n$,
$$[k;e_{k+1},  \ldots,e_{n-1},  e_n+n] = \sum_{s=0}^{n-1}(-1)^{n+s-1}[k;e_{k+1},  \ldots,e_{n-1}, e_n+s]Q_{n,s}^{p^{e_n}}.$$
\end{prop}

Using Theorem \ref{dl1} and Proposition \ref{md2} we explicitly compute the action of $St^{S,R}$ on M\`ui invariant $M_{n,s_1,\ldots,s_k}$ when $S,\ R$ are special. Particularly, we prove

\begin{thm}[H\uhorn ng and Minh \cite{hum}]\label{dl3} For $s_0 = -1 < s_1 < \ldots <s_k <  s_{k+1}= n$,
\begin{align*}
P^t&M_{n,s_1,\ldots,s_k} = \\
&\begin{cases}M_{n,t_1,\ldots,t_k}, &t = \underset{i=1}{\overset{k}\sum}\frac{p^{s_i}-p^{t_i}}{p-1}, \text{ with } s_{i-1} < t_i \leqslant s_i,\\
\underset{i=1}{\overset{k+1}\sum}(-1)^{k+1-i}M_{n,t_1,\ldots,\hat t_i\ldots, t_{k+1}}Q_{n,t_i}, &t = \underset{i=1}{\overset{k+1}\sum}\frac{p^{s_i}-p^{t_i}}{p-1}, \text{ with } s_{i-1} < t_i \leqslant s_i,\\ &1 \leqslant i \leqslant k+1, \ t_{k+1} < s_{k+1}=n,\\ 0, &\text{otherwise. }
\end{cases}
\end{align*}
\end{thm}

\section*{Acknowledgment} 
 I would like to thank Professor Hu\`ynh M\`ui for his generous help and inspiring guidance.

\section{Proof of Theorem \ref{dl1}}

First we recall M\`ui's results on the homomorphism $d_m^*P_m$ and the operations $St^{S,R}$.

Let $\mathcal A_{p^m}$ be the alternating group on $p^m$ letters. Suppose that $X$ is a topological space, $W\mathcal A_{p^m}$ is a contractible $\mathcal A_{p^m}$-free space. Then we have the Steenrod power map
$$P_m: H^q(X) \longrightarrow H^{p^mq}\big(W\mathcal A_{p^m}\underset{\mathcal A_{p^m}}\times X^{p^m}\big),$$
which sends $u$ to $1\otimes u^{p^m}$ at the cochain level (see [6; Chap. VII]).

The inclusion $(\mathbb Z/p)^n \subset \mathcal A_{p^m}$ together with the diagonal map $X\to X^{p^m}$ and the K\"unneth formula induces the homomorphism
$$d_m :H^{*}\big(W\mathcal A_{p^m}\underset{\mathcal A_{p^m}}\times X^{p^m}\big) \longrightarrow H^*(B(\mathbb Z/p)^n)\otimes H^*(X).$$
Set $\tilde M_{m,s} = M_{m,s}L_m^{h-1}, \ 0\leqslant s < m,\ \tilde L_m = L_m^h, \ h = (p-1)/2$. We have
\begin{thm}[M\`ui {[3; 1.3]}]\label{dlm} Let $u\in H^q(X),\ \mu(q) = (-1)^{hq(q-1)/2}(h!)^q$. Then
$$d_m^*P_m(u) = \mu(q)^m\sum_{S,R}(-1)^{r(S,R)}\tilde M_{m,s_1}\ldots \tilde M_{m,s_t}\tilde L_m^{r_0}Q_{m,1}^{r_1}\ldots Q_{m,m-1}^{r_{m-1}}\otimes St^{S,R}u.$$
Here the summation runs over all $(S,R)$ with $S = (s_1,\ldots,s_t)$, $0\leqslant s_1 < \ldots < s_t<m$, $R= (r_1, \ldots, r_m)$, $r_0 = q-t-2(r_1+\ldots +r_m)\geqslant 0$, $r(S,R) = t + s_1 +\ldots + s_t + r_1 + 2r_2 + \ldots + mr_m$.
\end{thm}
\begin{prop}[M\`ui \cite{mui1,mui2}]\label{md3}\

\smallskip
{\rm i)} $d_m^*P_m$ is a natural homomorphism preserving cup product up to a sign. Precisely,
$$d_m^*P_m(uv) = (-1)^{mhqr}d_m^*P_mud_m^*P_mv,$$
with $q =\dim u,\ r = \dim v$.

{\rm ii)} $d_m^*P_my_i = \underset{s=0}{\overset{m}\sum}(-1)^{m+s}Q_{m,s}\otimes y_i^{p^s}$.

\eject
{\rm iii)} $d_m^*P_m(x_1\ldots x_n) =$

$\mu(n)^m\underset{0\leqslant s_1 < \ldots < s_t<m}\sum (-1)^{t(n-t) + r(S,0)}\tilde M_{m,s_1}\ldots \tilde M_{m,s_t}\tilde L_m^{n-t}\otimes [n-t,s_1, \ldots ,s_t].$

\medskip\noindent Here $x_i$ and $y_i$ are defined as in the introduction.
\end{prop}
\begin{lem}\label{bd3} If $e_i \ne e_j$ for $i \ne j$, then
\begin{multline*}d_m^*P_m[e_{1},\ldots , e_n]\\ 
= \sum_{J = (J_0,\ldots,J_m)}(-1)^{mn + r(\emptyset,R_J)}\tilde L_m^{2r_{J_0}}Q_{m,1}^{r_{j_1}}\ldots Q_{m,m-1}^{r_{j_{m-1}}}\\ \otimes [e_1 + \Phi_J(1),\ldots,e_n + \Phi_J(n)],
\end{multline*}
where $R_J$ and $\Phi_J$ are defined as  in Theorem \ref{dl1}.
\end{lem}
\begin{proof} Let $\Sigma_n$ be the symmetric group on $n$ letters. Then
$$[e_{1},\ldots ,e_n] = \sum_{\sigma \in \Sigma_n}\text{sign}\ \! \sigma\prod_{i=1}^ny_i^{p^{e_{\sigma(i)}}}.$$
From Proposition \ref{md2}, we have
\begin{align*}d_m^*P_m\Big(\prod_{i=1}^ny_i^{p^{e_{\sigma(i)}}}\Big) &=\prod_{i=1}^n \big(d_m^*P_my_i\big)^{p^{e_{\sigma(i)}}} \\ & =\prod_{i=1}^n \Big( \underset{s=0}{\overset{m}\sum}(-1)^{m+s}Q_{m,s}^{p^{e_{\sigma(i)}}}\otimes y_i^{p^{e_{\sigma(i)}+s}}\Big).
\end{align*}
Expanding this product and using the definitions of $\Phi_J, R_J$ and the assumption of the lemma, we get
$$d_m^*P_m\Big(\prod_{i=1}^ny_i^{p^{e_{\sigma(i)}}}\Big) =  \sum_{J }(-1)^{mn + r(\emptyset,R_J)}Q_{m,0}^{r_{J_0}}\ldots Q_{m,m-1}^{r_{j_{m-1}}} \otimes \prod_{i=1}^n  y_i^{p^{e_{\sigma(i)}+\Phi_J(\sigma(i))}}.$$
Hence, from the above equalities we obtain
\begin{align*}d_m^*P_m&[e_{1},\ldots , e_n]\\
&= \sum_{J }(-1)^{mn + r(\emptyset,R_J)}Q_{m,0}^{r_{J_0}}\ldots Q_{m,m-1}^{r_{j_{m-1}}} \otimes\sum_{\sigma\in \Sigma_n}\text{sign}\ \! \sigma \prod_{i=1}^n  y_i^{p^{e_{\sigma(i)}+\Phi_J(\sigma(i))}} \\
&= \sum_{J }(-1)^{mn + r(\emptyset,R_J)}Q_{m,0}^{r_{J_0}}\ldots Q_{m,m-1}^{r_{j_{m-1}}} \otimes [e_1 + \Phi_J(1),\ldots,e_n + \Phi_J(n)].
\end{align*} 
Since $Q_{m,0} = \tilde L_m^2$, the lemma is proved.
\end{proof}

\begin{proof}[{\bf 2.4. Proof of Theorem \ref{dl1}}] Let $I$ be a subset of $\{1,\ldots, n\}$ and $I'$ is its complement in $\{1,\ldots, n\}$. Writing $I = (i_1,\ldots, i_k)$ and $I' = (i_{k+1}, \ldots , i_n)$ with $i_1<\ldots < i_k$ and $i_{k+1}< \ldots < i_n$. We set $x_I = x_{i_1}\ldots x_{i_k}$, $[e_{k+1},\ldots,e_n]_I = [e_{k+1},\ldots,e_n](y_{i_{k+1}},\ldots,y_{i_n})$ and $\sigma_I = \begin{pmatrix}1 &\ldots &n\\ i_1&\ldots &i_n \end{pmatrix}\in \Sigma_n$. In [2; I.4.2], M\`ui showed that
$$[k;e_{k+1},\ldots , e_{n-1}, e_n] = \sum_{I}\text{sign}\ \!\sigma_I x_I [e_{k+1},\ldots,e_n]_I.$$
From Proposition \ref{md3} and Lemma \ref{bd3} we have
\begin{multline*} d_m^*P_m(x_I) = \mu(k)^m\underset{0\leqslant s_1 < \ldots < s_t<m}\sum (-1)^{t(k-t) + r(S,0)}\tilde M_{m,s_1}\ldots \tilde M_{m,s_t}\tilde L_m^{k-t}\\ \otimes [k-t,s_1, \ldots ,s_t]_I,
\end{multline*}
where $[k-t,s_1, \ldots ,s_t]_I =  [k-t,s_1, \ldots ,s_t](x_{i_1},\ldots ,x_{i_k},y_{i_1},\ldots ,y_{i_k})$,
\begin{multline*} d_m^*P_m[e_{k+1},\ldots,e_n]_I\\ = \sum_{J=(J_0,\ldots, J_m)}(-1)^{m(m-k) + r(\emptyset,R_J)}\tilde L_{m}^{2r_{J_0}}Q_{m,1}^{r_{j_{1}}}\ldots Q_{m,m-1}^{r_{j_{m-1}}} \otimes\\ [e_{k+1} + \Phi_J(k+1),\ldots,e_n + \Phi_J(n)]_I.
\end{multline*}
Set $q = \dim [k;e_{k+1},\ldots , e_n]  = k + 2(p^{e_{k+1}}+\ldots +p^{e_n}).$ An easy computation shows that $\mu(q) = (-1)^{n-k}\mu(k)$ and $r(S,0) + r(\emptyset,R) = r(S,R)$. Hence from Proposition \ref{md3} and the above equalities we get
\begin{multline*} d_m^*P_m[e_{k+1},\ldots,e_n]\\ = \mu(q)^m\sum_{S,J}(-1)^{t(t-k) + r(S,R_J)}M_{m,s_1}\ldots \tilde M_{m,s_t}\tilde L_{m}^{k-t+ 2r_{J_0}}Q_{m,1}^{r_{j_{1}}}\ldots Q_{m,m-1}^{r_{j_{m-1}}} \otimes\\ \sum_I\text{sign}\ \!\sigma_I[k-t,s_1, \ldots ,s_t]_I[e_{k+1} + \Phi_J(k+1),\ldots,e_n + \Phi_J(n)]_I.
\end{multline*}
Then, using the Laplace development we obtain
\begin{multline*} d_m^*P_m[e_{k+1},\ldots,e_n]\\ = \mu(q)^m\sum_{S,J}(-1)^{t(t-k) + r(S,R_J)}M_{m,s_1}\ldots \tilde M_{m,s_t}\tilde L_{m}^{k-t+ 2r_{J_0}}Q_{m,1}^{r_{j_{1}}}\ldots Q_{m,m-1}^{r_{j_{m-1}}} \otimes\\ [k-t,s_1, \ldots ,s_t,e_{k+1} + \Phi_J(k+1),\ldots,e_n + \Phi_J(n)].
\end{multline*}
Theorem \ref{dl1} now follows this equality and Theorem \ref{dlm}.
\end{proof}

\section{Proof of Proposition \ref{md2}}

First we prove the stated relation for $k = 0$,
\begin{align}
\label{ct6}[e_{1},\ldots , e_{n-1}, e_n+n] = \sum_{s=0}^{n-1}(-1)^{n+s-1}[e_{1},\ldots,e_{n-1}, e_n+s]Q_{n,s}^{p^{e_n}}.
\end{align}
We will prove \eqref{ct6} and the following relation together by induction on $n$,
\begin{align}&\notag[e_{1},\ldots,e_{n-1}, e_n+n-1] = \sum_{s=0}^{n-2}(-1)^{n+s}[e_{1},\ldots,e_{n-1}, e_n+s]Q_{n-1,s}^{p^{e_n}} \\ 
 &\label{ct5} \hskip6cm  + [e_{1},\ldots, e_{n-1}]V_n^{p^{e_n}}.
\end{align}
Here, $V_n = L_n/L_{n-1}$.

We denote \eqref{ct6} and \eqref{ct5} when $n = m$ by \ref{ct6}$(m)$ and \ref{ct5}$(m)$, respectively.

When $n=2$ the proof is straightforward. Suppose that $n > 2$ and that \ref{ct6}$(n-1)$ and \ref{ct5}$(n-1)$ are true.

By Laplace development and \ref{ct5}$(n-1)$ we have
\begin{align*}
&[e_{1},\ldots,e_{n-1}, e_n+n-1]\\ &= \sum_{t=1}^{n-1}(-1)^{n+t}[e_{1},\ldots,\hat e_t,\ldots, e_{n-1}, e_n+n-1] y_n^{p^{e_t}} + [e_1,\ldots,e_{n-1}]y_n^{p^{e_n+n-1}}  \\
& = \sum_{t=1}^{n-1}(-1)^{n+t}\Big(\sum_{s=0}^{n-2}(-1)^{n+s}[e_{1},\ldots,\hat e_t,\ldots, e_{n-1}, e_n+s]Q_{n-1,s}^{p^{e_n}}\Big) y_n^{p^{e_t}}\\
&\hskip7cm + [e_1,\ldots,e_{n-1}]y_n^{p^{e_n+n-1}}\\
&=\sum_{s=0}^{n-2}(-1)^{n+s}\Big( \sum_{t=1}^{n-1}(-1)^{n+t}[e_{1},\ldots,\hat e_t,\ldots, e_{n-1}, e_n+s]y_n^{p^{e_t}}\Big)Q_{n-1,s}^{p^{e_n}} \\
&\hskip7cm + [e_1,\ldots,e_{n-1}]y_n^{p^{e_n+n-1}}\\
& =\sum_{s=0}^{n-2}(-1)^{n+s}[e_{1},\ldots,e_{n-1}, e_n+s]Q_{n-1,s}^{p^{e_n}}\\
&\hskip4cm  + [e_1,\ldots,e_{n-1}]\sum_{s=0}^{n-1}(-1)^{n+s-1}Q_{n-1,s}^{p^{e_n}}y_n^{p^{e_n+s}}.
\end{align*}
Since $V_n=\sum_{s=0}^{n-1}(-1)^{n+s-1}Q_{n-1,s}y_n^{p^{s}}$ (see \cite{dic}, \cite{mui1}), \ref{ct5}$(n)$ is proved.

Now we prove \ref{ct6}$(n)$. From \ref{ct5}$(n)$ and the relation $Q_{n,s}=Q_{n-1,s-1}^p + Q_{n-1,s}V_n^{p-1}$ (see \cite{dic}, \cite{mui1}) we obtain
 \begin{align*}
[e_{1},&\ldots , e_{n-1}, e_n+n]\\
 &= \sum_{s=1}^{n-1}(-1)^{n+s-1}[e_{1},\ldots,e_{n-1}, e_n+s]Q_{n-1,s-1}^{p^{e_n+1}}\\
&\hskip3cm+ [e_1,\ldots,e_{n-1}]V_n^{p^{e_n+1}}
 \end{align*}
 \begin{align*}
&= \sum_{s=1}^{n-1}(-1)^{n+s-1}[e_{1},\ldots,e_{n-1}, e_n+s]Q_{n,s}^{p^{e_n}}\\
&\hskip3cm - [e_1,\ldots,e_{n-1},e_n+n-1]V_n^{(p-1)p^{e_n}}\\
& + \Big(\sum_{s=1}^{n-2}(-1)^{n+s}[e_1,\ldots,e_{n-1},e_n+s]Q_{n-1,s}^{p^{e_n}} \\
&\hskip3cm+ [e_1,\ldots,e_{n-1}]V_n^{p^{e_n}}\Big)V_n^{(p-1)p^{e_n}}.
\end{align*}

Combining this equality and \ref{ct5}$(n)$ we get
\begin{align*}
[e_{1},e_2,\ldots , e_{n-1}, e_n+n]
 &= \sum_{s=1}^{n-1}(-1)^{n+s-1}[e_{1},\ldots,e_{n-1}, e_n+s]Q_{n,s}^{p^{e_n}}\\
&- (-1)^{n}[e_1,\ldots,e_{n-1},e_n]Q_{n-1,0}^{p^{e_n}} V_n^{(p-1)p^{e_n}}.
\end{align*}
Since $Q_{n,0}=Q_{n-1,0}V_n^{p-1}$, the proof of \ref{ct6}$(n)$ is completed.
 
For $0 < k < n$, Proposition \ref{md2} follows from \eqref{ct6} and [2; I.4.7] which asserts that 
\begin{align*}&[k;e_{k+1},\ldots, e_n] =\\ & \ (-1)^{k(k-1)/2}\sum_{0\le s_1<\ldots <s_k}(-1)^{s_1+\ldots +s_k}M_{n,s_1, \ldots,s_k}[s_1, \ldots, s_k, e_{k+1} ,\ldots, e_n]/L_n.
\end{align*}
The proposition is completely proved.

\section{Some applications}

In this section, using Theorem \ref{dl1} and Proposition \ref{md2}, we prove Theorem \ref{dl3} and explicitly compute the action of $St^{S,R}$ on M\`ui invariant $M_{n,s_1,\ldots,s_k}$ when $S, R$ are special. First we prove Theorem \ref{dl3}. 
\begin{proof}[{\bf 4.1. Proof of Theorem \ref{dl3}}] Recall that $P^t = St^{\emptyset,(t)}$. From Theorem \ref{dl1} we have
\begin{align*} P^t&M_{n,s_1,\ldots,s_k} \\
&= \begin{cases} [k; 0,\ldots,\hat t_1,\ldots,\hat t_{k+1},\ldots, n], &t= \underset{i=1}{\overset{k+1}\sum}\frac{p^{s_i}-p^{t_i}}{p-1}, \text{ with }\\ & s_{i-1} < t_i \leqslant s_i, 1 \leqslant i \leqslant k+1,\\ 0, &\text{otherwise. }
\end{cases}
\end{align*}
If $t_{k+1} = s_{k+1} = n$, then $[k; 0,\ldots,\hat t_1,\ldots,\hat t_{k+1},\ldots, n] = M_{n,t_1,\ldots,t_k}$. Suppose $t_{k+1} < n$. By Proposition \ref{md2} we have
\begin{align*} [k; 0,\ldots,\hat t_1,&\ldots,\hat t_{k+1},\ldots, n] \\
&= \sum_{s=0}^{n-1}(-1)^{n+s-1}[k; 0,\ldots,\hat t_1,\ldots,\hat t_{k+1},\ldots,n-1, s]Q_{n,s}\\ &=  \sum_{s=0}^{n-1}(-1)^{k+1-i}M_{n,t_1,\ldots,\hat t_i,\ldots, t_{k+1}}Q_{n,t_i}
\end{align*}
Hence Theorem \ref{dl3} follows.
\end{proof}
\setcounter{thm}{1}
\begin{nota}\label{kh2} Denote by $S' : s_{k+1} <\dots < s_{n-1}$ the ordered complement  of a sequence $S: 1 \leqslant s_1<\ldots < s_k < n$ in $\{1, \ldots,n-1\}$. Set $\Delta_i = (0,\ldots, 1, \ldots, 0)$ with 1 at the $i$-th place $(1\leqslant i \leqslant n)$, $\Delta_0 = (0,\ldots, 0)$ and $R = (r_1,\ldots,r_n)$. Here, the length of $\Delta_i$ is $n$.
\end{nota}

The following was proved in M\`ui [3; 5.3] for $R= \Delta_0$.
\begin{prop}\label{md4} Set $s_0 = 0$. Under the above notations, we have
$$St^{S',R}M_{n,1,\ldots,n-1} = \begin{cases} (-1)^{(k-1)(n-1-k) + s_t-t}M_{s_0,\ldots,\hat s_t,\ldots,s_k}, &R = \Delta_{s_t},\\ \underset{t=0}{\overset{k}\sum}(-1)^{k(n-k) -t}M_{s_0,\ldots,\hat s_t,\ldots,s_k}Q_{n,s_t}, &R = \Delta_{n}, \\ 0, &\text{otherwise. }\end{cases} 
$$
\end{prop}
\begin{proof} Note that $M_{n,1,\ldots,n-1} = [n-1;0]$. From Theorem \ref{dl1} we obtain
$$St^{S',R}M_{n,1,\ldots,n-1} = \begin{cases} (-1)^{k(n-1-k)}[k;1,\ldots,\hat s_1,\ldots,\hat s_k,\ldots,n-1,i],\\
 \hskip1cm R = \Delta_{i},\text{ with } i = s_t,\ 0\leqslant t \leqslant k, \text{ or } i = n,\\
 0, \hskip.7cm\text{otherwise. }\end{cases} $$
It is easy to see that
$$[k;1,\ldots,\hat s_1,\ldots,\hat s_k,\ldots,n-1,s_t] = (-1)^{n-1-k+s_t-t}M_{s_0,\ldots,\hat s_t,\ldots,s_k}.$$
According to Proposition \ref{md2} we have
\begin{align*}[k;1,\ldots,&\hat s_1,\ldots,\hat s_k,\ldots,n-1,n]\\
 &= \sum_{s=0}^{n-1} (-1)^{n + s -1}[k;1,\ldots,\hat s_1,\ldots,\hat s_k,\ldots,n-1,s]Q_{n,s}\\
&= \sum_{t=0}^{k}(-1)^{k-t} M_{s_0,\ldots,\hat s_t,\ldots,s_k}Q_{n,s_t}.
\end{align*}
From this the proposition follows.
\end{proof}

By the same argument as given in the proof of Theorem \ref{dl3} and Proposition \ref{md4} we obtain the following results.

\begin{prop}\label{md5} Let $\Delta_i$ be as in \ref{kh2} and $s_0 = 0$. Then
$$St^{\emptyset,\Delta_i}M_{n,s_1,\ldots,s_k} = \begin{cases} (-1)^{s_t-t}M_{s_0,\ldots,\hat s_t,\ldots,s_k}, &s_1 > 0,\ i = s_t,\\ \underset{t=0}{\overset{k}\sum}(-1)^{n-t-1}M_{s_0,\ldots,\hat s_t,\ldots,s_k}Q_{n,s_t}, &s_1>0,\ i=n,\\
 0, &\text{otherwise.}\end{cases} 
$$
\end{prop}

The following proposition was proved by H\uhorn ng and Minh \cite{hum} for $s = 0$.

\begin{prop}\label{md6} For $0 \leqslant s \leqslant n$,
$$St^{(s),(0)}M_{n,s_1,\ldots,s_k} = \begin{cases} (-1)^{k+s_t-t}M_{s_0,\ldots,\hat s_t,\ldots,s_k}, &s= s_t,\\ 
\underset{i=1}{\overset{k}\sum}(-1)^{n+k+t+1}M_{s_1,\ldots,\hat s_t,\ldots,s_k}Q_{n,s_t}, &s=n,\\
 0, &\text{otherwise.}\end{cases} 
$$
\end{prop}

\bigskip
\bibliographystyle{amsplain}

\end{document}